\documentclass[]{amsart}
\usepackage{amssymb,amsmath, graphicx}
\usepackage{mathdots}
\usepackage{mathrsfs}
\usepackage{wrapfig}
\usepackage{enumerate}
\usepackage{url}
\usepackage[utf8]{inputenc}

\makeatletter
\DeclareFontFamily{OMX}{MnSymbolE}{}
\DeclareSymbolFont{MnLargeSymbols}{OMX}{MnSymbolE}{m}{n}
\SetSymbolFont{MnLargeSymbols}{bold}{OMX}{MnSymbolE}{b}{n}
\DeclareFontShape{OMX}{MnSymbolE}{m}{n}{
    <-6>  MnSymbolE5
   <6-7>  MnSymbolE6
   <7-8>  MnSymbolE7
   <8-9>  MnSymbolE8
   <9-10> MnSymbolE9
  <10-12> MnSymbolE10
  <12->   MnSymbolE12
}{}
\DeclareFontShape{OMX}{MnSymbolE}{b}{n}{
    <-6>  MnSymbolE-Bold5
   <6-7>  MnSymbolE-Bold6
   <7-8>  MnSymbolE-Bold7
   <8-9>  MnSymbolE-Bold8
   <9-10> MnSymbolE-Bold9
  <10-12> MnSymbolE-Bold10
  <12->   MnSymbolE-Bold12
}{}

\let\llangle\@undefined
\let\rrangle\@undefined
\DeclareMathDelimiter{\llangle}{\mathopen}%
                     {MnLargeSymbols}{'164}{MnLargeSymbols}{'164}
\DeclareMathDelimiter{\rrangle}{\mathclose}%
                     {MnLargeSymbols}{'171}{MnLargeSymbols}{'171}
\makeatother

\def\at#1#2{
#1\hskip 0.25ex\vline_{\hskip 0.25ex\raisebox{-1.5ex}
{{$\scriptstyle#2$}}}}

\def\textmap#1{\mathop{\vbox{\ialign{
                                  ##\crcr
      ${\scriptstyle\hfil\;\;#1\;\;\hfil}$\crcr
      \noalign{\kern 1pt\nointerlineskip}
      \rightarrowfill\crcr}}\;}}
\def\bigtextmap#1{\mathop{\vbox{\ialign{
                                  ##\crcr
      ${\hfil\;\;#1\;\;\hfil}$\crcr
      \noalign{\kern 1pt\nointerlineskip}
      \rightarrowfill\crcr}}\;}}

\def\textlmap#1{\mathop{\vbox{\ialign{
                                  ##\crcr
      ${\scriptstyle\hfil\;\;#1\;\;\hfil}$\crcr
      \noalign{\kern-1pt\nointerlineskip}
      \leftarrowfill\crcr}}\;}}

\def\at#1#2{
#1\hskip 0.25ex\vline_{\hskip 0.25ex\raisebox{-1.5ex}
{{$\scriptstyle#2$}}}}

\def\textmap#1{\mathop{\vbox{\ialign{
                                  ##\crcr
      ${\scriptstyle\hfil\;\;#1\;\;\hfil}$\crcr
      \noalign{\kern 1pt\nointerlineskip}
      \rightarrowfill\crcr}}\;}}
\def\bigtextmap#1{\mathop{\vbox{\ialign{
                                  ##\crcr
      ${\hfil\;\;#1\;\;\hfil}$\crcr
      \noalign{\kern 1pt\nointerlineskip}
      \rightarrowfill\crcr}}\;}}

\def\textlmap#1{\mathop{\vbox{\ialign{
                                  ##\crcr
      ${\scriptstyle\hfil\;\;#1\;\;\hfil}$\crcr
      \noalign{\kern-1pt\nointerlineskip}
      \leftarrowfill\crcr}}\;}}

\def\R{{\mathbb R}}
\def\Z{{\mathbb Z}}

\def\ig{{\mathfrak i}}

\def\kg{{\mathfrak k}}

\theoremstyle{remark}

\theoremstyle{plain}

\newtheorem{sz}{Satz}[section]
\newtheorem{thry}[sz]{Theorem}
\newtheorem{pr}[sz]{Proposition}

\newtheorem{co}[sz]{Corollary}
\newtheorem{dt}[sz]{Definition}
\newtheorem{lm}[sz]{Lemma}

\theoremstyle{remark}
\newtheorem{re}{Remark}

\theoremstyle{plain}

\def\End{\mathrm {End}}

\def\U{\mathrm{U}}

\def\id{ \mathrm{id}}

\def\ad{\mathrm {ad}}
\def\Diff{\mathrm {Diff}}
\def\U2{\mathrm{U(2)}}
\def\niq{=\kern-.18cm /\kern.08cm}

\def\ad{\mathrm{ad}}

\newcommand{\cal}{\mathcal}

\begin{document}

\title[Infinitesimal isometries of connection metric]{Infinitesimal isometries of connection metric and generalized moment map equation}
\author{Arash Bazdar}
\address{Aix Marseille Université, CNRS, Centrale Marseille, I2M, UMR 7373, 13453 Marseille, France}
\email{arash.bazdar@univ-amu.fr}

\begin{abstract}
Let $(M,g)$ be a smooth Riemannian manifold, $K$ a compact Lie group and $p:P\to M$ a principal $K$-bundle over $M$ endowed with a connection $A$. Fixing a bi invariant inner product on Lie algebra $\kg$ of $K$, the connection $A$ and metric $g$ define a Riemannian metric $g_A$ on $P$. Let $\tilde X$ be the horizontal lift of vector field $X$ on $M$ and, let $\xi^\nu$ be the vertical field associated with section $\nu\in A^0(\ad( P))$ of the adjoint bundle. It is proved that the connection $A$ is invariant under the 1-parameter group of local diffeomorphism generated by $\tilde X+\xi^\nu$ if and only if $X$ and $\nu$ satisfy the generalized moment map equation $\iota_XF_A=-\nabla^A\nu$. The Lie algebra of fiber preserving Killing fields of $(P,g_A)$ is studied, in the case where $K$ is compact, connected and semisimple. 
\end{abstract}
\maketitle
\section{Lifts associated to a section of the adjoint bundle }\label{introd}
\subsection{Introduction.} \label{TheProblem} 
Let $(M,g)$ be a differentiable Riemannian manifold and $K$ be a compact Lie group. Let $p:P\to M$ be a principal $K$-bundle on $M$ and $A$ a connection on $P$. The same notation $A$ is used to denote the horizontal distribution associated to connection $A$. Fixing an inner product $\llangle\cdot,\cdot\rrangle$ on the Lie algebra $\kg$ of $K$ we can define a Riemannian metric $g_A$ on $P$ characterized by the following condition 
\begin{enumerate}
\item If $V_P:=\ker p_*$ denotes the vertical bundle, then the canonical bundle isomorphism $V_P\simeq P\times \kg$ is an orthogonal bundle isomorphism with respect to the inner products defined by $g_A$ and $\llangle\cdot,\cdot\rrangle$.
\item The restriction of $p_*:T_P\to T_M$ to the horizontal subbundle $A\subset T_P$ gives an orthogonal bundle isomorphism $A\to p^*(T_M)$.
\item By the definition, the horizontal distribution $A\subset T_P$ is $K$-invariant and the following short exact sequence is splitting. 
$$
\{0\}\to V_P\to T_P\to A\to \{0\}.
$$
We require that the direct sum decomposition $T_P\simeq A\oplus V_P$ of the tangent bundle $T_P$ is $g_A$-orthogonal. 

\end{enumerate}
The metric $g_A$ on $P$ defined in this way is called the connection metric and defines a Riemannian submersion $p:(P,g_A)\to (M,g)$ with totally geodesic fibers \cite{Vi}. Moreover, if $g$ is a complete metric on $M$, then the connection metric $g_A$ is also complete \cite{Zi}. We use $I(M)$ to denote the group of isometries of $(M,g)$. It is well known that the isometry group of a compact manifold is a finite dimensional Lie group \cite{KN}. The group of $K$-equivariant covering bundle isomorphisms of $P$ is defined by
$$I_K(P):=\{(\Phi,\varphi)|\ \varphi\in I(M),\ \Phi:P\to P\hbox{ is a $\varphi$-covering bundle isomorphism}\}.
$$  
The group $I_K(P)$ has a natural  topology (induced by the weak ${\cal C}^\infty$-topology \cite[section 2.1]{Hi}).
A map $\Phi:P\to P$ is called $\varphi$-covering isomorphism of principal $K$-bundles, if it is a $K$-equivariant diffeomorphism of $P$ such that the induced map $\varphi:M\to M$ on the base manifold is a diffeomorphism satisfying $\, p\circ \Phi= \varphi \circ p$. Let $\Diff^A_K(P)$ denote the space of $K$-equivariant diffeomorphism of $P$ leaving connection $A$ invariant, $I(P)$ the group of isometries of $(P,g_A)$ and  $I^A_K(P)$ the stabilizer of connection $A$ in the group $I_K(P)$, then it is not difficult to prove that (Lemma \ref{RaisomP}) 
$$ 
I^A_K(P)=\Diff^A_K(P)\cap I(P).
$$
 This Lie group play an important role in theory of locally homogeneous triples introduced in \cite {Ba}. A bundle map $\Phi:P\to P$ is called fibre preserving if for all $x\in M$, it restricts to smooth map $\Phi_x:P_x\to P_{\Phi(x)}$. If   $I_V(P)$ denotes the Lie group of all fiber preserving isometries of $(P,g_A)$ then we have $I^A_K(P)\subset I_V(P)\subset I(P)$. Use $\mathfrak i_K(P)$, $\mathfrak i_V(P)$ and $\mathfrak i(P)$ to denote the Lie algebras of $I^A_K(P)$, $I_V(P)$ and $I(P)$. Our goal is to describe the Lie algebra $\mathfrak i_V(P)$ of all fiber preserving infinitesimal isometries of $(P,g_A)$. If $K$ is compact, connected and semisimple, we will give a complete description of the structure of $\mathfrak i_V(P)$. If $P$ is compact, $I(P)$ is a finite dimensional Lie group and therefore $I^A_K(P)$ and $I_V(P)$ is also a finite dimensional. 
%
%
%
Any element of the Lie algebra $\mathfrak i_V(P)$ is an infinitesimal isometry (Killing vector field) of total space $P$ with respect to metric $g_A$. 
 \begin{dt} \label{HorKill}
 A vector field $Y\in {\cal X}(P)$ on a principal $K$-bundle $P$ is called fiber preserving if, the 1-parameter group of local diffeomorphisms generated by $Y$ maps each fiber into another fiber.   \end{dt}
\begin{re}
A vector field $Y\in {\cal X}(P)$ is fiber preserving if and only if the Lie bracket $[Y,Z]$ is vertical for any vertical vector field $Z\in{\cal X}^v(P)$.
 \end{re}
Let us recall some basic facts about infinitesimal isometries (Killing vector fields). A vector field $X\in {\cal X}(M)$ on Riemannian manifold $(M,g)$ is called a Killing vector field if, the $1$-parameter group of local diffeomorphisms generated by it in a neighborhood of each point of $M$ are the local isometries. This is also equivalent to the following conditions
\begin{enumerate}
\item $L_Xg=0$\,(where $L_X$ denotes the Lie derivative with respect to $X$).
\item For all vector fields $X_1,X_2\in {\cal X}(M)$
$$Xg(X_1,X_2)=g([X,X_1],X_2)+g(X_1,[X,X_2]).$$
\end{enumerate}
For a compact manifold $M$, the set of all Killing vector fields on $M$ is a finite dimensional Lie subalgebra of ${\cal X}(M)$ and can be considered as the Lie algebra of group of isometries $I(M)$ and will be denoted by $\mathfrak{i}(M)$. 
\subsection{ Connections invariant by lifts}
Let $M$ be a smooth manifold of dimension $\dim M=n$. A rank $r$ distribution on $M$ is a subbundle $D\subset T_M$ of the tangent bundle with constant rank $r$. In other words $D$ is a smooth distribution if and only if for any fixed point $x_0\in M$, there exist a family $\{X_1,...,X_r\}$ of vector fields defined on some neighborhood $U$ of $x_0$ such that $\{X_1(x),...,X_r(x)\}$ is a basis for $D_x$ at any point $x\in U$. We have the next result in the theory of distribution.
\begin{lm}\label{Distributioninvariant}
Let $M$ be a smooth $n$-dimensional manifold. Let $D\subset T_M$ be a rank $r$ distribution on $M$ and let $X$ be a vector field on $M$.  If for every vector field $Y\in\Gamma(D)$ one has $[X,Y]\in\Gamma(D)$, then the distribution $D$ is invariant under the $1$-parameter group of local diffeomorphism generated by $X$.
\end{lm}
\begin{proof}
Fix $x_0\in M$ and choose an open neighborhood $U\subset M$ of $x_0$ and $\epsilon>0$ such that $(-\epsilon,\epsilon)\times U$ is contained in the maximal domain of definition of the flow of the vector field $X$.  We may suppose that the subbundle $D_U$ is trivial over $U$. Hence, there exists a family $\{X_1,...,X_r\}$ of smooth vector fields defined on $U$ such that $\{X_1(x),...,X_r(x)\}$ is a basis for $D_x$ at any point $x\in U$. Let $\gamma_t$ denotes the flow of the vector field $X$. We need just to prove that for every $Y\in\Gamma(D)$, the local field $Y_t:=(\gamma_t)_*Y_{x}$ defined at every $(t,x)\in(-\epsilon,\epsilon)\times U$ is a section of $\Gamma(D)$.
Let $\{\omega_1,...,\omega_{n-r}\}$ be a family of independent $1$-forms on $U$ which define a basis of the annihilator of $\Gamma(D_U)$. Consider the functions
\begin{equation}\label{DLFORHA}
\varphi_j(t):=((\gamma_t)^*\omega_j)(Y_{x})=(\omega_j)_{\gamma_t(x)}(Y_t).
\end{equation}
Taking derivation with respect to $t$, we obtain
$$\frac{d\varphi_j(t)}{dt}=(L_X\omega_j)_{\gamma_t(x)}(Y_t)=X_{\gamma_t(x)}(\omega_j(Y))-\omega_j([X,Y])_{\gamma_t(x)}.
$$
Since $Y\in\Gamma(D)$ and $[X,Y]\in\Gamma(D)$, we have $\omega_j(Y)=\omega_j([X,Y])=0$. Hence $  d\varphi_j(t)/{ dt}=0$ which means that $\varphi_j(t)$ is constant. As $\varphi_j(0)=0$, the right hand side of (\ref{DLFORHA}) is vanishes and we have $Y_t\in\Gamma(D)$.
\end{proof}
\begin{pr}\label{InvariantLIFT}
Let $p:P\to M$ be a principal $K$-bundle over $n$-dimensional manifold $M$. For a vector field $Y\in {\cal X}(P)$  the following are equivalent
\begin{enumerate}
\item The horizontal distribution $A\subset T_P$ is invariant under the $1$-parameter group of local diffeomorphism generated by $Y$.
\item For every horizontal field $Z\in{{\cal X}}^h(P)$, the Lie bracket $[Y,Z]$ is horizontal. 
\item For every horizontal lift $\tilde X\in{\cal X}^h(P)$ of vector field $X\in {\cal X}(M)$, the Lie bracket $[Y,\tilde X]$ is horizontal. 
\end{enumerate}
\end{pr}

\begin{proof}
$(1)\Leftrightarrow  (2)$ follows by Lemma \ref{Distributioninvariant}. The implication $(2) \Rightarrow (3)$ is evident. To prove the implication $(3) \Rightarrow (2)$, let $Z\in {\cal X}^h(P)$, so it is a section of horizontal distribution $A\subset T_P$. There exists an open set $U\subset M$, a family of horizontal lifts $\tilde X_1,...,\tilde X_n$ on $P_U$ and a family of maps $\varphi_1,...,\varphi_n$ such that $Z=\varphi_1\tilde X_1+...+\varphi_n\tilde X_n$. So, we have
\begin{equation*}
[Y,Z]=[Y, \sum _{i=1}^n\varphi_i\tilde X_i]\\=\sum_{i=1}^n (Y(\varphi_i) \tilde X_i+\varphi_i[Y,\tilde X_i]).
\end{equation*}
Since $\tilde X_i\in{\cal X}^h(P)$ and $[Y,\tilde X_i]\in{\cal X}^h(P)$  we have $[Y,Z]\in{\cal X}^h(P)$.
\end{proof}
\begin{pr}\label{alphaAtiez}
  Let $p:P\to M$ be a principal $K$-bundle on $M$ endowed with a connection  $A$. If $a^\#$ denotes the fundamental field corresponding to $a\in\kg$ and  $\tilde X$ is the horizontal lift of $X\in{\cal X}(M)$, then the vector field $a^\#$ is invariant under the $1$-parameter group of local diffeomorphism generated by $\tilde{X}$.
\end{pr}
\begin{proof}
The vector field $a^\#$ at $y\in P$ is defined by
$$
a^\#_y:=\at{\frac{d}{dt}}{t=0} (y \exp ta ).
$$
Let $(\beta_t)$ be the $1$-parameter group of local diffeomorphism generated by $\tilde{X}$. The vector field $a^\#$ is invariant under $(\beta_t)$ if for all $(t,y)\in\R\times P$ for which $\beta_t$ is defined, one has
 $$ (\beta_t)_{*y}(a_y^\#)=a_{\beta_t(y)}^\#.
 $$
Using Proposition \ref{InvariantLIFT} the vector field $a^\sharp$ is invariant under $\beta_t$ if and only if $[\tilde{X},a^\sharp]=0$. Set $k_t=\exp(ta)\in K$ and use $R_{k_{t*}}\tilde{X}=\tilde{X}$ to obtain
$$[a^\#,\tilde X]=\lim_{t\to0}\frac{1}{t}(\tilde X-R_{k_{t*}}\tilde{X})=0.
$$%
\end{proof}
%
The canonical bundle isomorphism $V_P\simeq P\times \kg$ can be used to identify the Lie subalgebra of vertical fields ${\cal X}^v(P):=\Gamma(V_P)$ with the space of smooth map ${\cal C}^\infty(P, \kg)$. This identification is given by $\nu\mapsto \xi^\nu$, where $\xi^\nu$ is used to denote the vertical field associated with $\nu\in {\cal C}^\infty(P, \kg)$ and  $\xi^\nu:P\to V_P$ at a point $y\in P$ is defined by
$$\xi^\nu_y:=\at{\frac{d}{dt}}{t=0}(y \exp t\nu(y)).
$$
A smooth map $\nu:P\to \kg$ is called $K$-equivariant if, for every $(y,k)\in P\times K$ one has $\nu(yk)=\ad_{k^{-1}}\nu(y)$, where $\ad:K\to \End(\kg)$ denotes the adjoint representation. The space of all $K$-equivariant maps is denoted by ${\cal C}^K(P,\kg)$. Now, suppose $A^0(\ad (P))$ denotes the space of smooth sections of the adjoint bundle $\ad(P):=P\times_\ad \kg\to M$. There exists a natural identification between two sapces ${\cal C}^K(P,\kg)$ and $A^0(\ad (P))$, given at a point $y\in P_x$ ($x\in M$) by
$${\cal C}^K(P,\kg)\to A^0(\ad (P)), \ \nu\mapsto (x\mapsto [y,\nu(y)])
$$
By the above discussion it is clear that:
  \begin{co}\label{xinujadid}
  For each $K$-invariant vertical vector field $Y\in{\cal X}^v(P)$, there exists a section $\nu\in A^0(\ad (P))$ such that $Y=\xi^\nu$.
\end{co}
 \begin{lm}\label{XTILDEXINU}
Let $p:P\to M$ be a principal $K$-bundle endowed with connection $A$. If $\tilde X$ denotes the horizontal lift of a vector field $X$ on $M$, then for any  section $\nu\in A^0(\ad (P))$ we have $[\tilde{X},\xi^{\nu}]=\xi^{\nabla^A_X\nu}.$
\end{lm}
\begin{proof}
Let $t\mapsto x_t=x(t)$ be an integral curve of $X$ defined for $t\in(-\varepsilon,\varepsilon)$ for some $\varepsilon> 0$  with the initial condition $x(0)=x_0$. The horizontal lift $t\mapsto y_t$ of the path $t\mapsto x_t$ through a point $y_0\in P_{x_0}$ is the integral curve of $\tilde X$ with initial condition $y(0)=y_0$. A section $\nu\in A^0(\ad (P))$ can be regarded as a $K$-equivariant map $\nu:P\to\kg$. Put $a_t:=\nu(y_t)\in\kg$, for all $t\in(-\varepsilon,\varepsilon)$. The vertical vector field $\xi^\nu_{y_t}$ is given by
$$\xi^\nu_{y_t}=(a_t)^\#_{y_t}=\at{\frac{d}{ds}}{s=0}(y_t\exp(sa_t)).
$$ 
Let $(\beta_t)_{t\in(-\varepsilon,\varepsilon)}$ be the $1$-parameter group of local diffeomorphism generated by $\tilde{X}$. We have 
	\begin{equation*}
			[\tilde{X},\xi^\nu]_{y_0}=\lim_{t\rightarrow0}\frac{1}{t}\{\xi^\nu_{y_0}-(\beta_t)_*(\xi^\nu_{y_{-t}})\}=\lim_{t\rightarrow0}\frac{1}{t}\{(a_0)^\#_{y_0}-(\beta_t)_*(a_{-t})^\#_{y_{-t}}\}
	\end{equation*}
	Since $(\beta_t)_*(a_{-t})^\#_{y_{-t}}=(a_{-t})^\#_{y_0}$, we have
	\begin{equation*}
		\begin{split}
			[\tilde{X},\xi^\nu]_{y_0}&=\lim_{t\rightarrow0}\frac{1}{t}\{(a_0)^\#_{y_0}-(a_{-t})^\#_{y_0}\}=\{\lim_{t\rightarrow0}\frac{1}{t}(a_{0}-a_{-t})\}^\#_{y_0}
			=\{(\dot a_t)_{t=0}\}^\#_{y_0}\\
			&=\{\frac{ d}{dt}\nu (y_t)_{t=0}\}^\#_{y_0}=\{d\nu(\tilde{X}_{y_0})\}^\#_{y_0}=\{\nabla^A_{X_{y_0}}\nu\}^\#_{y_0}=\xi^{\nabla^A_{X_{y_0}}\nu}.
		\end{split}
	\end{equation*}

\end{proof}
\begin{lm}\label{XTILDEXINU}
	Let $p:P\to M$ be a principal $K$-bundle. For any $a\in\kg$ and any section $\nu\in A^0(\ad( P))$ we have $[a^\#,\xi^{\nu}]=0$. In particular, the vector field $\xi^\nu$ is $K$-equivariant.
\end{lm}
\begin{proof} Let $(\varphi_t)_{t\in\R}$  be the $1$-parameter group of  diffeomorphisms generated by the vector field $a^\#$. The maps $\varphi_t:P\to P$ are defiend by $y\mapsto y\,\exp(ta)$. Let $(f_s)_{s\in\R}$ be the $1$-parameter group of  diffeomorphisms generated by $\xi^\nu$. Using \cite[Proposition 1.11]{KN} the Lie bracket $[a^\#,\xi^\nu]=0$ if and only if  $\varphi_t\circ f_s=f_s\circ \varphi_t$ for all $s,t\in \R$. Since $f_s$ is a bundle isomorphism, it commutes with the right translations and for all $y\in P$
	\begin{equation*}
		f_s(\varphi_t(y))=f_s(y\exp(ta))=f_s(y)\exp(ta)=\varphi_t(f_s(y))\,.
	\end{equation*}
	\end{proof}
 Suppose $p:P \to M$ is a principal $K$-bundle over $M$ with connection $A$ and use $\nabla^A:A^0(\ad(P))\to A^1(\ad(P))$ to denote the induced connection on $\ad(P)$. The covariant exterior derivative associated with connection $A$ is denoted by $d^{\nabla^A}:A^r(\ad (P))\to A^{r+1}(\ad (P))$ and the curvature $2$-form $F_A \in A^2(\ad (P)) $ is defined by
  $$F_A=d^{\nabla^A}\circ \nabla^A.
  $$
Let $Y\in{\cal X}(P)$ be a vector field on $P$ and $(\beta_t)$ be the $1$-parameter group of local diffeomorphism generated by it. We say that the connection $A$ is invariant under the $1$-parameter group of local diffeomorphism generated by $Y$, if
$$(\beta_t)_*A_y=A_{\beta_t(y)}.
$$
for all $(y,t)\in P\times \R$ for which $\beta_t(y)$ is defined. 
  \begin{dt}
	Let  $p:P\to M$ be a principal $K$-bundle over $M$ endowed with a connection $A$. Let $X$ be a vector field on $M$ and $\nu\in A^0(\ad(P))$ be a section of the adjoint bundle. We say that $X$ and $\nu$ satisfy the generalized moment map equation if $\iota_XF_A=-\nabla^A\nu$. 
	\end{dt}
 It is known \cite[p. 257]{GH} that for two horizontal lifts $\tilde X_1, \tilde X_2$ of vector fields $X_1, X_2$ on $M$ we have the following formula 
\begin{equation}\label{CURVATUREFURMOLA}
[\tilde X_1,\tilde X_2]=\widetilde{[X_1,X_2]}-\xi^{F_A(X_1,X_2)}.
\end{equation}
Therefore, the horizontal projection of $[\tilde X_1,\tilde X_2]$ is $\widetilde{[X_1, X_2]}$ and the vertical projection of it is $[\tilde X_1,\tilde X_2]^v= -\xi^{F_A(X_1,X_2)}$.
\begin{thry}\label{momentap}
Let $p:P \to M$ be a principal $K$-bundle on $M$ with connection $A$ and the curvature $2$-form $F_A\in A^2(\ad(P))$. Let $\tilde X$ be the horizontal lift of a vector field $X$ on $M$ and $\nu\in A^0(\ad(P))$. The connection $A$ is invariant under the $1$-parameter group of local diffeomorphism generated by $\tilde X+\xi^\nu$ if and only if $X$ and $\nu$ satisfy the generalized moment map equation.
\end{thry}
\begin{proof}
By Proposition \ref{InvariantLIFT} the connection $A$ is invariant under the $1$-parameter group of local diffeomorphism generated by $\tilde X+\xi^\nu$ if and only if for every horizontal lift $\tilde Z$ of vector field $Z$ on $M$, the Lie bracket $[\tilde X+\xi^\nu,\tilde Z]$ is also horizontal. Using equation (\ref{CURVATUREFURMOLA}) we obtain
\begin{align*}
[\tilde X+\xi^\nu,\tilde Z]=[\tilde X,\tilde Z]+[\xi^\nu,\tilde Z]=\widetilde {[X,Z]}-\xi^{F_A(X,Z)}+[\xi^\nu,\tilde Z].
\end{align*}
The connection $A$ is invariant under the $1$-parameter group of local diffeomorphism generated by $\tilde X+\xi^\nu$ if and only if $[\tilde X+\xi^\nu,\tilde Z]\in\Gamma(A)$. But, this happen if and only if
\begin{equation}\label{xiOmegZ}
\xi^{F_A(X,Z)}=[\xi^\nu,\tilde Z].
\end{equation}
By Lemma \ref{XTILDEXINU}, $[\xi^\nu,\tilde Z]=-\xi^{\nabla^A\nu(Z)}$. Hence (\ref{xiOmegZ}) is equivalent to $\iota_XF_A=-\nabla^A\nu.$ 
\end{proof}
\begin{pr}\label{betaleaveinvariantV}
Let $\tilde{X}$ be the horizontal lift of a vector field $X$ on $M$ and let $\nu\in A^0(\ad(P))$. If $(\beta_t)$ denotes the $1$-parameter group of local diffeomorphism generated by $\tilde{X}+\xi^\nu$, then $(\beta_t)$ leaves invariant the vertical distribution.
\end{pr}
\begin{proof}
Using Proposition \ref{Distributioninvariant} the vertical bundle $V_P$ is invariant by $(\beta_t)$ iff for any section $Z\in\Gamma(V_P)$ one has $[\tilde{X}+\xi^\nu,Z]\in\Gamma(V_P)$. But
\begin{equation*}
[\tilde X+\xi^\nu,Z]=[\tilde X,Z]+[\xi^\nu,Z].
\end{equation*}
The integrability of vertical distribution imply $[\xi^\nu,Z]\in\Gamma(V_P)$. Moreover, $[\tilde X,Z]$ is also a vertical section. To see this note that the space of vertical sections is a free ${\cal C}^\infty(P)$-module, generated by the fundamental vector fields. Thus, for any vertical section $Z\in\Gamma(V_P)$ there exists a family$\{a_1^\#,...,a^\#_r\}$  of fundamental vector fields and a family $\{f_1,...,f_r\}$ of smooth functions on $P$ such that $Z=f_1a_1^\#+...+f_ra_r^\#$. Therefore
\begin{equation*}
\begin{split}
[\tilde X,Z]=[\tilde X,\sum_{i=1}^rf_ia_i^\#]&=\sum_{i=1}^r(  df_i (\tilde X) a_i^\#+f_i[\tilde X,a_i^\#])=\sum_{i=1}^r  df_i (\tilde X) a_i^\#\in\Gamma(V_P).
\end{split}
\end{equation*}
\end{proof}
\section{Fiber preserving infinitesimal isometries of connection metric }
Let $p:P\to M$ be a principal $K$-bundle endowed with connection $A$. 
 Let $g$ be a Riemannian metric on $M$ and $\llangle\cdot,\cdot\rrangle$ be a $\ad$-invariant metric on the Lie algebra $\kg$ of $K$. We summarize these information by saying that $(g,P\textmap{p} M,A)$ is a $K$-triple over $M$. For $y\in P$ the fibre over $x\in M$ is denoted by $P_x:=p^{-1} (x)$. For all $y\in P_x$ the map $\tau_y:K\to P_x$ defined by $\tau_y(k):=yk$ is a diffeomorphism and can be used to define a metric on the fiber $P_x$. To show that this metric is well defined, let $y'\in P_x$ be another point in the fibre, then there exists $k_0\in K$ such that $y'=yk_0$. For all $k\in K$
$$\tau_{yk_0}(k)=(yk_0)k=y(k_0k)=\tau_y (l_{k_0}k)=(\tau_y l_{k_0})(k).
$$
where $l_{k_0}$ denotes the left translation by $k_0$. Since the metric on $K$ is $\ad$-invariant the induced metric on $(V_P)_y\simeq T_yP_x$ is well defined. Use the linear isomorphism $A_y\simeq T_xM$ to lift the metric $g_x$ on the base to a metric on the horizontal space $A_y$. These two metrics define a metric $g_A$ on $T_P\simeq A\oplus V_P$.

If $\omega_A\in A^1(P, \mathfrak{k})$ denotes the connection $1$-form associated to connection $A$, then $\omega_A$ and $\ad$-invariant metric $\llangle\cdot,\cdot\rrangle$ on $\kg$ define a symmetric tensor $\llangle\omega_A,\omega_A\rrangle$ of type $(0,2)$ acts on $(v,w)\in T_yP\times T_yP$ by
$$\llangle\omega_A,\omega_A\rrangle(v,w):=\llangle\omega_A(v),\omega_A(w)\rrangle,
$$
 Adding $p^*g$ to $\llangle\omega_A,\omega_A\rrangle$ we obtain a Riemannian metric $g_A:=p^*g +\llangle\omega_A,\omega_A\rrangle$ on total space $P$ such that for all vectors $v$, $w \in T_yP$ 
$$g_A(v,w)=g\left(p_{*}(v),p_{*}(w)\right)+\llangle\omega_A(v),\omega_A(w)\rrangle.
$$
It is clear that $g_A$ is the connection metric on $P$.
\begin{lm}\label{RaisomP}
Suppose $(g,P\textmap{p} M,A)$ is a $K$-triple over $M$, then
\begin{enumerate}[(a)]
\item The right multiplication $R_k:P\to P$ is an isometry of $P$ with respect to metric $g_A$. 
\item Let $\Phi\in \Diff^K_A(P)$ and denote the induced map on the base by $\varphi \in \Diff(M)$. Then,  $\Phi\in I(P)$ if and only if $\varphi\in I(M)$
\item $I^A_K(P)=\Diff^A_K(P)\cap I(P).$
\end{enumerate}
\end{lm}
\begin{proof} (a) Since the inner product $\llangle\cdot, \cdot \rrangle$ is $\ad$-invariant and for all $k\in K$: $p\circ R_k=p$\,, $R^*_{k}\omega_A=\ad_{k^{-1}}\omega_A$ we have
\begin{equation*}\label{omegstar}
R_k^*g_A=R_k^*(p^*g)+R_k^*\llangle \omega_A,\omega_A\rrangle=p^*g+\llangle \ad_{k^{-1}}\omega_A,\ad_{k^{-1}}\omega_A\rrangle=g_A\,.
\end{equation*}
(b) Suppose $\Phi:P\to P$ is a $\varphi$-covering bundle isomorphism which leaves invariant the connection $A$, then $p\circ\Phi=\varphi\circ p$ and $\Phi^*\omega_A=\omega_A$. If $\varphi\in I(M)$, then 
\begin{equation*}
\begin{split}
\Phi^*g_A&=\Phi^*(p^*g)+\Phi^*(\llangle\omega_A,\omega_A\rrangle)=(p\circ \Phi)^*g+\llangle \Phi^*\omega_A,\Phi^* \omega_A\rrangle\\
&=p^*(\varphi^*g)+\llangle \omega_A, \omega_A\rrangle=p^*g+\llangle \omega_A, \omega_A\rrangle=g_A.\\
\end{split}
\end{equation*}
Conversely, if $\Phi\in I(P)$ then $\Phi^*g_A= g_A$, and hence for all vector fields $X, Y$ on $M$ and their horizontal lifts $\tilde X, \tilde Y\in {\cal X}^h(P)$ 
\begin{equation*}
\begin{split}
g(X,Y)&=g(p_*(\tilde X),p_*(\tilde Y))=(p^*g)(\tilde X,\tilde Y)=g_A(\tilde X,\tilde Y)=(\Phi^*g_A)(\tilde X, \tilde Y)\\
&=\Phi^*(p^*g)(\tilde X, \tilde Y)=(p^*\varphi^*)g(\tilde X,\tilde Y)=(\varphi^*g)( X, Y).
\end{split}
\end{equation*}
(c) It follows from (b) and the definition of $I^A_K(P)$.
\end{proof}

\begin{re} \label{aSharpIsKilling}
For any $a\in\kg$, the map $R_{\exp ta}:P\to P$ is the $1$-parameter group of diffeomorphism generated by $a^\#$. It is clear that $a^\#\in \mathfrak i_V(P)$. 
\end{re}
\begin{pr}\label{KillingnU}
Let $(g,P\textmap{p} M,A)$ be a $K$-triple over $M$. Let $\tilde{X}$ denote the horizontal lift of vector field $X\in {\cal X}(M)$ and $\nu\in A^0(\ad (P))$. The vector field $\tilde{X}+\xi^\nu$ is Killing on $(P,g_A)$ if and only if $X$ is a Killing field on $(M,g)$ and $\iota_XF_A=- \nabla^A\nu.$
\end{pr}
\begin{proof}
 Let $(\beta_t)$ be the $1$-parameter group of local diffeomorphism generated by $Y:=\tilde X+\xi^\nu$, and let $(\alpha_t)$ denote the $1$-parameter group of local diffeomorphism generated by $X$. If $X$ is Killing and $\iota_XF_A=-\nabla^A\nu$, then $(\alpha_t)$ are local isometries of $(M,g)$ and the connection $A$ is invariant under $(\beta_t)$ using Theorem \ref{momentap}. From Lemma \ref{RaisomP} we deduce that $(\beta_t)$ are local isometries of $(P,g_A)$ which means that $Y$ is Killing. Conversely, if $Y=\tilde X+\xi^\nu$ is Killing then $X$ is Killing and we have $\beta_t^*g_A=g_A$ for all $t$. By Proposition \ref{betaleaveinvariantV} the vertical distribution is invariant under $(\beta_t)$, hence the horizontal distribution is invariant under $(\beta_t)$ which is equivalent to the generalized moment map equation $\iota_XF_A=-\nabla^A\nu$ by Theorem \ref{momentap}. 
 \end{proof}
\begin{re}
Let ${\cal G}(P)$ denote the gauge group of principal $K$-bundle $p:P \to M$ \cite{DK}, \cite{Te}, i.e. the group of $\id_M$-covering bundle automorphisms of $P$. The space of smooth sections of adjoint bundle $A^0(\ad (P))$ with point wise bracket can be considered as the Lie algebra of gauge group \cite{Bl}. If ${\cal G}^A(P)$ denotes the stabilizer of connection $A$ in the gauge group ${\cal G}(P)$, then ${\cal G}^A(P)\subset I^A_K(P)$. Moreover, thanks to Proposition \ref{KillingnU} the Lie algebra of ${\cal G}^A(P)$ is given by 
$$\mathrm{Lie}({\cal G}^A(P))=\{ \nu\in A^0(\ad(P)) |\, \nabla^A\nu=0\}.
$$  
 \end{re}

\begin{re}
Using Proposition \ref{KillingnU}, if for a Killing vector field $X\in \ig(M)$ there exists a section $\nu_0\in A^0(\ad(P))$ which satisfies in the generalized moment map equation, then $\tilde X+\xi^{\nu_0}$ is a fiber preserving Killing field on $(P,g_A)$.  
  \end{re}
	Suppose that $K$ is a connected, compact and semisimple Lie group endowed with a bi invariant metric, then any Killing vector field on $K$ is a sum of a right invariant vector field and a left invariant vector field on $K$ \cite{OT}.
\begin{lm} \label{VerKill}
Let $(g,P\textmap{p} M,A)$ be a $K$-triple over $M$. If $K$ is connected, compact and semisimple, then any vertical Killing field $Y$ on $(P,g_A)$ is decomposed as $Y=\xi^\nu+a^\#$, where $\nu\in A^0(\ad(P))$ is a $\nabla^A$-parallel section and $a\in \kg$.
\end{lm}
\begin{proof} For all $y\in P$ the map $\tau_y:K\to P_{p(y)}$ defined by $k\mapsto yk$ is an isometry. Denote by $T_y:\kg\to T_{y}P_{p(y)}\simeq (V_P)_y$  the derivation of $\tau_y$ at identity element of $K$. Since $Y$ is Killing and $\tau_y$ is isometry $T_y^{-1}(Y_y)\in \kg$ is also Killing for all $y\in P$. Using the fact that $K$ is connected, compact and semisimple, we have the following decomposition
$$T_y^{-1}(Y)=v_1(y)+v_2(y).
$$
where $v_1(y)$ is a right invariant vector field and $v_2(y)$ is a left invariant vector field on $K$. Put $Y_1:=T_y(v_1(y))$, $Y_2:=T_y(v_2(y))$. One can check that $Y_1$, $Y_2$ are well defined vertical field on $P$. We will show that there exist a section $\nu\in A^0(\ad(P))$ and an element $a\in\kg$ such that $Y_1=\xi^\nu$, $Y_2=a^\#$. Let $\nu:P\to \kg$ be a map defined at $y\in P$ by $\nu(y):= T_y^{-1}(Y_1)$. For any $k\in K$ we have $\tau_{yk}=\tau_y l_k$, where $l_k:K\to K$ denotes the left translation by $k\in K$. If we denote the derivation of $l_k$  by $L_k:\kg\to\kg$ then $T_{yk}=T_y\circ L_k$ and hence
\begin{equation*}
\begin{split}
\nu(yk)=T_{yk}^{-1}(Y_1)&=L_k^{-1}(T_y^{-1}(Y_1))\\
=L_k^{-1}(v_1(y))=&(L_k^{-1}R_k) v_1(y)=\ad_{k^{-1}}\nu(y).
\end{split}
\end{equation*}
Hence $\nu \in {\cal C}^K(P,\kg)\simeq A^0(\ad(P))$. So, there exist a section $\nu\in  A^0(\ad(P))$ such that $Y_1=\xi^\nu$. The vertical vector field $\xi^\nu$ is Killing if and only if for all vector fields $Z_1,Z_2 \in {\cal X}(P)$ one has
$$\xi^\nu g_A(Z_1,Z_2)=g_A([\xi^\nu,Z_1], Z_2)+g_A(Z_1,[\xi^\nu,Z_2]).
$$
Taking $Z_1=\tilde X$ and $Z_2=c^\#$ where $X$ is an arbitrary vector field on $M$ and $c\in \kg$ and using Lemma \ref{XTILDEXINU} we see that $\nabla^A_X\nu=0$, which means that $\nu$ is $\nabla^A$ parallel.
If $y,y'\in P$ are two arbitrary points in the same fiber, then there exist an element $k\in K$ such that $y'=yk$. This implies $\tau_{y'}=\tau_y\circ l_k$ and $T_{y'}=T_y\circ L_k$, so
\begin{equation*}
v_2(y')=T^{-1}_{y'}(Y_2(y'))=(L_k^{-1}T^{-1}_y)(Y_2(y))=L_k^{-1}(v_2(y))=v_2(y).
\end{equation*}
whence $v_2(y)$ is constant on the fiber and there exist $a\in\kg$ such that $Y_2=a^\#$.
 \end{proof}
 \begin{thry}\label{main2}
Let $K$ be a connected, compact and semisimple Lie group. Let $(g,P\textmap{p} M,A)$ be a $K$-triple over $M$. If $Y\in \ig_V(P)$ is a fiber preserving Killing vector field on $(P,g_A)$, then
\begin{enumerate}[(a)]
\item There exists a unique Killing vector field $X\in \ig(M)$ such that the horizontal projection of $Y$ is the horizontal lift $\tilde X$ of $X$. 
\item If the generalized moment map equation for $X$ admits a solution, then there exists a section $\nu\in A^0(\ad(P))$ and $a\in \kg$ such that $Y$ decompose as 
	$$Y=\tilde X+\xi^\nu+a^\# \ \ \ \hbox{where}\ \ \ \iota_XF_A=-\nabla^A\nu.
	$$
\end{enumerate}
\end{thry}
 \begin{proof}  (1) Decompose $Y$ into horizontal projection $Y^h$ and vertical projection $Y^v$. Since $Y^v$ is fiber preserving, so is $Y^h$, hence the Lie bracket $[Y^h,a^\#]$ is vertical. But the Lie bracket of a horizontal field and a fundamental field is horizontal \cite[Page. 65]{KN}. Therefore $[Y^h,a^\#]=0$ for all $a\in \kg$ which means that $(R_k)_*(Y^h)=Y^h$, for all $k\in K$. Using  \cite[Proposition 1.2]{KN}  there exist a unique vector field $X$ on $M$ such that $Y^h=\tilde X$. Because $Y$ is Killing and $p_*(Y)=X$, the vector field $X$ is also Killing. (2) If $X$ is Killing and there exist a section $\nu_0\in A^0(\ad(P))$ such that $ \iota_XF_A=-\nabla^A\nu_0$, then using Proposition \ref{KillingnU} we conclude that  $\tilde X+\xi^{\nu_0}$ is Killing. Therefore, $Y-(\tilde X+\xi^{\nu_0})$ is a vertical Killing field and Lemma \ref{VerKill} implies that there exist a $\nabla^A$-parallel section $\nu_1\in A^0(\ad(P))$ and an element $a\in \kg$ such that 
 $$Y-(\tilde X+\xi^{\nu_0})=\xi^{\nu_1}+a^\#.
 $$ 
 Putting $\nu:=\nu_0+\nu_1$ we obtain $Y=\tilde X+\xi^\nu+a^\#,$
 and we have $\iota_XF_A=-\nabla^A\nu$.
 \end{proof}
  %
%
%

\begin{pr}
 Let $(g,P\textmap{p} M,A)$ be a $K$-triple over $M$ with compact Lie group $K$. The Lie algebra of $K$-invariant Killing fields of $(P,g_A)$, denoted by $\mathfrak{i}_K(P)$, is a Lie subalgebra of $\mathfrak{i}(P)$ and is given by
$$\mathfrak{i}_K(P)=\{\tilde X+\xi^\nu|\, \iota_XF_A=-\nabla^A\nu \ ,\ \hbox{for  $(X,\nu)\in\ig(M)\times A^0(\ad (P))$}\}.
$$  
 	\end{pr}
\begin{proof}
 	If $Y\in \mathfrak{i}_K(P)$, then by Theorem \ref{main2} there exists a unique Killing vector field $X\in \ig(M)$ such that the horizontal projection $Y^h$ is equal to $\tilde X$. As $Y$ and $Y^h$ are $K$-invariant Killing fields, the vertical projection $Y^v$ is also $K$-invariant. Since the action of $K$ on $P$ is free the map $\nu:P\to \kg$ is equivariant if and only if $\xi^\nu$  is invariant. Using corollary \ref{xinujadid} there exists a section $\nu\in A^0(\ad (P))$ such that $Y=\xi^\nu$. Hence, $Y=\tilde X+\xi^\nu$ and Proposition \ref{KillingnU} implies that $\iota_XF_A=-\nabla^A\nu$. Conversely, for a Killing field $X\in \ig(M)$ and a section $\nu\in A^0(\ad (P))$ which satisfy the generalized moment map equation, we have $\tilde X+\xi^\nu\in\ig_K(P)$.

 	\end{proof}
\section{Examples and applications}
\subsection{Fiber preserving Killing vector field of orthogonal frame bundle} Let $(M,g)$ be an oriented Riemanian manifold and $P=\mathrm{SO}(M)$ its principal bundle of orthonormal frame. Let $A$ be the Levi-Civita connection on $P$ and endow it by connection metric $g_A$. The induced linear connection on associated bundle $T_M\simeq P\times_{\mathrm{SO}(n)}\R^n$ coincide with the Levi-Civita  connection of metric $g$ and is denoted by $\nabla:A^0(T_M)\to A^1(T_M)$. The connection $A$ induces a linear connection $\nabla^A$ on the adjoint bundle $\ad(P)=\End(T_M)$. The Riemann curvature tensor is $F_A\in A^2(\End (T_M))$. If $X\in\ig(M)$ is a Killing vector field on $M$ then using \cite[Lemma 6.1]{Sa} we have $\iota_XF_A=-\nabla^{A}\nabla X$. Thanks to Proposition \ref{KillingnU}, it follows that the vector field $X^L:=\tilde X+\xi^{\nabla X}$ is Killing and will be called the natural lift of $X$. We have the following theorem.
\begin{pr} \cite[Theorem A.]{TaYa} Let $\mathrm{SO}(M)$ be the frame bundle of the oriented Riemannian manifold $(M,g)$. If $Y$ is a fiber preserving Killing vector field on $(\mathrm{SO}(M),g_A)$, then $Y$ can be decomposed as 
$$Y=X^L+\xi^\nu+a^\#.
$$
where $X^L$ is the natural lift of a Killing field $X$ on $M$, $\nu$ is a $\nabla^A$-parallel section of $\End(T_M)$ and $a\in \mathfrak{so}(n)$.
\end{pr}
\begin{proof} As $Y$ is a fiber preserving Killing field on $\mathrm{SO}(M)$, there exists a Killing field $X\in \ig(M)$ such that $Y^h=\tilde X$. Using \cite[Lemma 6.1]{Sa}, since $X\in\ig(M)$ is Killing we have $\iota_XF_A=-\nabla^{A}\nabla X$. Regarding $\nabla X$ as a section of $\End(T_M) \simeq \ad(\mathrm{SO}(M) ) $, the generalized moment map equation $\iota_XF_A=-\nabla^A\nu$ admits a solution $\nu_0:=\nabla X\in A^0(\End(T_M))$. Using Proposition \ref{KillingnU} we deduce that the vector field $X^L:=\tilde X+\xi^{\nabla X}$ is Killing and Theorem \ref{main2} complete the proof.  
\end{proof}
\subsection{ Quantiziable symplectic manifold}
A symplectic manifold $(M,\omega)$ is called quantiziable if and only if there is a principal circle bundle $p:P\to M$ over $M$ and a 1-form $\alpha\in A^1(P,\R)$ which is invariant under the action of $\mathbb S^1$ and $d\alpha=p^*\omega$. One can prove that $(M,\omega)$ is quantiziable if and only if
$$\frac{1}{2\pi}[\omega]\in H^2(M,\Z).
$$
For a proof of this result see \cite[Page. 440]{AM}. By the classical Chern-Weil theory there exist a connection $A$ on $P$ for which the connection form is given by $\theta_A=\alpha$ and the curvature form is $F_A=\omega $. The circle bundle $P$ is called the quantizing manifold. In this case, the space of sections of the adjoint bundle $A^0(\ad (P))$ can be identified with ${\cal C^\infty}(M)$. Let $\nu \mapsto f_\nu$ denote this identification. If $\tilde{X}$ denotes the horizontal lift of a Killing field $X$ on $M$ and $\nu\in A^0(\ad (P))$, then using Proposition \ref{KillingnU} we deduce that the vector field $\tilde X+\xi^\nu$ is Killing with respect to $g_A$ if and only if $X$ and $\nu$ satisfy the generalized moment map equation. This is to say that $X$ is a Hamiltonian vector field with Hamiltonian function $H:M\to\R$ given by $H=-f_\nu+c$, where $c\in\R$ is a constant. 

\end{document}